\documentclass{amsart}%
\usepackage{amssymb}
\usepackage{color}
\usepackage{amsmath}
\usepackage{graphicx}
\usepackage{amsfonts}%
\newcommand{\EE}{\ensuremath{\mathbb{E}}}
\newcommand{\FF}{\ensuremath{\mathbb{F}}}

\newcommand{\NN}{\ensuremath{\mathbb{N}}}

\newcommand{\PP}{\ensuremath{\mathbb{P}}}

\newcommand{\RR}{\ensuremath{\mathbb{R}}}

\renewcommand{\phi}{\varphi}

\newcommand{\eps}{\ensuremath{\epsilon}}

\newcommand{\bB}{\ensuremath{\mathcal{B}}}

\newcommand{\fF}{\ensuremath{\mathcal{F}}}

\newtheorem{theorem}{Theorem}

\newtheorem{corollary}[theorem]{Corollary}

\newtheorem{definition}[theorem]{Definition}

\newtheorem{lemma}[theorem]{Lemma}

\newtheorem{remark}[theorem]{Remark}

\title[L\'evy areas of Ornstein--Uhlenbeck processes in Hilbert--spaces]
{L\'evy--areas of Ornstein--Uhlenbeck processes in Hilbert--spaces}

\author{Mar\'{\i}a J. Garrido-Atienza}
\address[Mar\'{\i}a J. Garrido-Atienza]{Dpto. Ecuaciones Diferenciales y An\'alisis Num\'erico\\
Universidad de Sevilla, Apdo. de Correos 1160, 41080-Sevilla,
Spain} \email[Mar\'{\i}a J. Garrido-Atienza]{mgarrido@us.es}

\author{Kening Lu}
\address[Kening Lu]{346 TMCB\\
Brigham Young University, Provo, UT 84602, USA} \email[Kening
Lu]{klu@math.byu.edu}

\author{Bj{\"o}rn Schmalfu{\ss }}
\address[Bj{\"o}rn Schmalfu{\ss }]{Institut f\"{u}r Stochastik\\
Friedrich Schiller Universit{\"a}t Jena, Ernst Abbe Platz 2, 77043\\
Jena,
Germany
 }
\email[Bj{\"o}rn Schmalfu{\ss }]{bjoern.schmalfuss@uni-jena.de}

\subjclass{Primary: ;  Secondary: }
\keywords{Stochastic PDEs, Hilbert-valued fractional Brownian motion, path--wise solution.}

\begin{document}
\begin{abstract}
In this paper we investigate the existence and some useful properties of the L\'evy areas of Ornstein--Uhlenbeck processes associated to Hilbert--space-valued fractional Brownian--motions with Hurst parameter $H\in (1/3,1/2]$.  We prove that this stochastic area has a H{\"o}lder--continuous version with sufficiently large H{\"o}lder--exponent and that can be approximated by smooth areas. In addition, we prove the stationarity of this area.
\end{abstract}
\maketitle

\section{Introduction}

During the last decades new techniques have been developed generalizing the well known Ito--or Stratonovich--integration.
For fundamental publications in this area see for instance Lyons an Qian \cite{Lyons} for the so called {\em Rough--Paths} theory and Z{\"a}hle \cite{Zah98}
for the {\em Fractional Calculus}\footnote{The name fractional calculus theory goes back to D. Nualart.} theory.
In particular, these techniques allow to have stochastic integrators which are more general than the Brownian motion.
A candidate for such an integrator is for instance the fractional Brownian motion. This stochastic process does not have in general the semi--martingale property,
which would allow to define the stochastic integral as a limit in probability if the integrator satisfies special measurability and integrability conditions,
see e.g. Karatzas and Shreve \cite{KarShr91}. Another advantage of these new techniques is to treat stochastic integrals in a path--wise way.
In particular, for any sufficient regular integrand and integrator with more general properties than the bounded variation property, an integral can be defined.
This kind of integrals goes back to Young \cite{You36} which allows to consider integrals for H{\"o}lder--continuous integrands and integrators fulfilling special conditions with respect to the H{\"o}lder--exponents. In practice, this integral cannot be used to replace the well known Ito--integral
by a Young--integral if the integrator is a Brownian motion and the integrand has a H{\"o}lder--exponent less than or equal to $1/2$.

Our reason for dealing with this type of new integral is to take advantage of the path--wise property in order to further introduce {\em Random Dynamical Systems} (RDS) for infinite dimensional differential equations, namely, for stochastic evolution equations and stochastic partial differential equations as well. The random driver of this kind of equation will be in general a trace class fractional Brownian--motion with Hurst--parameter $H\in (1/3,1/2]$.
To introduce an RDS one needs at first a model for a noise which is called a {\em metric dynamical system}: Consider the quadruple $(\Omega,\fF,\PP,\theta)$
where $(\Omega,\fF,\PP)$ is a probability space and $\theta$ is a {\em measurable flow} on $\Omega$:
\begin{align*}
  &\theta: (\RR\times \Omega,\bB(\RR)\otimes \fF)\to (\Omega,\fF) \\
  & \theta_t\circ\theta_\tau=\theta_t\theta_\tau=\theta_{t+\tau},\qquad \theta_0={\rm id}_\Omega.
\end{align*}
An Hilbert--(or topological) space RDS is a measurable mapping
\begin{equation*}
  \phi: (\RR^+\times\Omega\times V,\bB(\RR^+)\otimes \fF\otimes \bB(V))\to (V,\bB(V))
\end{equation*}
satisfying the cocycle property
\begin{equation*}
  \phi(t+\tau,\omega,x)=\phi(t,\theta_\tau\omega,\phi(\tau,\omega,x)),\qquad \phi(0,\omega,x)=x,
\end{equation*}
for all $t,\,\tau\in \RR^+$, $x\in V$ and {\em for all} $\omega\in\Omega$ or at least for all $\omega$ of a $\theta$--invariant set $\tilde \Omega \in \fF$ of full $\PP$--measure which is independent of $x,\,t,\,\tau$. This is different to the fact that {\em a property holds almost surely} for any $x\in V$, which is often used  in Stochastic Analysis, since exceptional sets depending on $t,\,\tau$ and $x$ are in most cases forbidden when dealing with the cocycle property.

\smallskip

An example of a metric dynamical system is for instance $(C_0,\bB(C_0),\PP_H, \theta)$, where the probability space $(C_0,\bB(C_0),\PP_H)$ is the canonical probability space such that $C_0$ is the space of continuous functions on $\RR$ with values in a separable Hilbert--space $V$, $\bB(C_0)$ is its Borel--$\sigma$--algebra, $\PP_H$ is the distribution on $\bB(C_0)$ of a trace class fractional Brownian--motion $\omega$ with Hurst--parameter $H\in (0,1)$, and $\theta: (\RR\times C_0,\bB(\RR)\otimes \bB(C_0))\to (C_0,\bB(C_0))$ is given by
\begin{equation}\label{shift}
\theta_\tau(t)=\omega(t+\tau)-\omega(\tau), \quad \text{for } t,\tau\in \RR,
\end{equation}
see Arnold \cite{Arn98} or Maslowski and Schmalfuss \cite{BB}.

\smallskip

It is known that under typical conditions on the coefficients an ordinary Ito--equation generates an RDS, see Arnold \cite{Arn98}.
The main instrument to obtain this property is the Kolmogorov--theorem on the existence of a H{\"o}lder--continuous version of a random field.
Such a random field is derived from an Ito--equation where the parameters of this field are the time and the non-random initial condition.
The so called perfection technique then can be used to conclude the existence of a version of the ordinary Ito--equation which defines an RDS. However, there is no appropriate version of the Kolmogorov--theorem for infinite dimensional random fields that could be applied to show that solutions of stochastic evolution equations generate an RDS, and therefore this property is rather an open problem, although there are partial results for particular cases, see e.g. the recent papers \cite{BGM}, \cite{CKSch}, \cite{FL}, \cite{GGSch} and \cite{GLSch}.

\bigskip

Consider a stochastic evolution equation
\begin{equation*}
  du=Audt+G(u)d\omega
\end{equation*}
where $A$ is the generator of an analytic stable semigroup $S$ on the separable Hilbert--space $V$ and $\omega$ is a $\beta$--H{\"o}lder--continuous fractional--Brownian motion in $V$ with Hurst--parameter $H\in (1/3,1/2]$, so that $\beta \in (1/3,1/2)$. This equation has the mild interpretation
\begin{equation}\label{steeq3}
  u(t)=S(t)u_0+\int_0^tS(t-r)G(u(r))d\omega(r)
\end{equation}
where $u_0$ is a non-random initial condition in $V$. The integral in the above equation has to be interpreted in a fractional sense. For a good understanding of that integral we refer to \cite{HuNu} and \cite{GLSch-local}.

For the following let $\bar\Delta_{a,b}\subset\RR^2$ be the set of the pairs $(s,t)$ such that $-\infty<a\le s\le t\le b<\infty$. Let $V\otimes V$ be the Hilbert--tensor space
of $V$ with tensor product $\otimes_V$. We consider now functions
\begin{equation*}
  \bar\Delta_{0,T}\ni (s,t)\mapsto (u\otimes\omega)(s,t).
\end{equation*}
The reason to introduce these elements is to
interpret the integral of \eqref{steeq3} in a fractional sense. We need to consider the tensor product of a possible solution $u$ and a noise path $\omega$. In particular, for a smooth $\omega$ this tensor product is given by
\begin{align}\label{ns}
\begin{split}
    (u\otimes\omega)(s,t)=&\int_s^t(S(\xi-s)-{\rm id})u(s)\otimes_V\omega^\prime(\xi)d\xi\\
    &+\int_s^t\int_s^\xi S(\xi-r)G(u(r))\omega^\prime(r)dr\otimes_V\omega^\prime(\xi)d\xi,
    \end{split}
\end{align}
and, exchanging the order of integration, the last integral of \eqref{ns} can be written as
\begin{equation*}
  \int_s^t G(u(r))D_1(\omega\otimes_S\omega)(r,t)dr\footnotemark
\end{equation*}
\footnotetext{$D_1$ means the first derivative w.r.t. the first variable of $(\omega\otimes_S\omega)(\cdot,\cdot)$.}
where $(\omega\otimes_S\omega)(s,t)$  is defined by
\begin{align}\begin{split}
  L_2(V,\hat V)\ni E\mapsto E(\omega\otimes_S\omega)(s,t)&=\int_s^t\int_s^\xi
    S(\xi-r)E\omega^\prime(r) \otimes_V \omega^\prime(\xi)drd\xi\\
   &=\int_s^t\int_r^t
   S(\xi-r)E\omega^\prime(r) \otimes_V \omega^\prime(\xi)d\xi dr
   \end{split}\label{omegaSS}
\end{align}
for $(s,t)\in\bar\Delta_{0,T}$.  Here $\hat V$ is another Hilbert--space to be determined later and $E$ is an element of the Hilbert--Schmidt space $L_2(V,\hat V)$.
Let $C_{2\beta}(\bar \Delta_{a,b},L_2(L_2(V,\hat V),V\otimes V))$ the space of $2\beta$--H{\"o}lder--continuous fields on $\bar\Delta_{a,b}$
with values in $V\otimes V$ with norm
\begin{equation}\label{steeq102}
 \|v\|_{2\beta}=\|v\|_{2\beta,a,b}= \sup_{(s,t)\in\bar\Delta_{a,b}}\frac{\|v\|_{L_2(L_2(V,\hat V),V\otimes V)}}{|t-s|^{2\beta}}<\infty.
\end{equation}

We know that $\omega\in C_{\beta}([0,T];V)$ for every $\beta<H$. We also denote by $\omega^n$ a piecewise linear (continuous) approximation of $\omega$ with respect to an equidistant partition of length $2^{-n}T:=\delta$ such that $\omega^n(t)=\omega(t)$ for the partition points $t$. For these $\omega^n$ we can define
$(\omega^n\otimes_S\omega^n)$ by the right hand side of \eqref{omegaSS}.\\

The main purpose of this paper is to prove the following result, a property which is needed to establish the existence of solutions to (\ref{steeq3}). For a detailed description on the construction of solutions to (\ref{steeq3}) we refer the reader to the paper \cite{GLSch-local}.

\begin{theorem}\label{stet1}
Let $(\omega^n)_{n\in\mathbb{N}}$ be the sequence  of piecewise linear approximations of some $\omega$ introduced above such that
$((\omega^n\otimes_S\omega^n))_{n\in\mathbb{N}}$ is defined by
\eqref{omegaSS}. Then for any $\beta<H$ the sequence
$((\omega^n,(\omega^n\otimes_S\omega^n)))_{n\in\mathbb{N}}$ converges to
$(\omega,(\omega\otimes_S\omega))$ in
$C_{\beta}([0,T];V)\times C_{2\beta} (\bar \Delta_{0,T}; L_2(L_2(V,\hat V),V\otimes V))$ on a set of full measure.
In particular $(\omega\otimes_S\omega)$ is continuous.
\end{theorem}

In the following we consider only the statements of this theorem with respect to $(\omega\otimes_S\omega)$. The convergence properties
of $(\omega^n)_{n\in\mathbb{N}}$ follow in a similar and simpler manner. We omit here.

\smallskip

The object $(\omega\otimes_S\omega)$ will be called {\it L\'evy area of the Ornstein-Uhlembeck process.} Since a priori it is not clear whether $(\omega\otimes_S\omega)$ is well defined, in what follows we are going to give an appropriate meaning to this term by presenting two proofs of the above theorem. The first one is related to $\beta$-H{\"o}lder continuous paths
covering the case of a fractional Brownian--motion for an appropriate Hilbert--space $\hat V$. The second proof deals with the case of a Brownian-motion for a more general space $\hat V$.\\

As we have said, $A$ is the  generator of the analytic semigroup $S$ on the separable Hilbert--space $V$. We also suppose that $-A$ is positive and symmetric such that its inverse is compact.
Then $-A$ has a positive point spectrum $(\lambda_i)_{i\in\NN}$ of finite multiplicity such that $\lim_{i\to\infty}\lambda_i=\infty$.
We denote by $(e_i)_{i\in\NN}$ the associated eigenelements which form a complete orthonormal system in $V$.
Furthermore, $-A$ generates for $\kappa\in\RR$ the separable Hilbert--spaces
\begin{equation*}
  D((-A)^\kappa)=:V_\kappa=\{u=\sum_i\hat u_ie_i:\|u\|_{V_\kappa}^2=\sum_i|\hat u_i|^2\lambda_i^{2\kappa}<\infty\},
\end{equation*}
with $V=V_0$.

\section{The construction of $(\omega\otimes_S\omega)$ for a fractional Brownian--motion}\label{stes2}

Let us begin this section by introducing a $V$-valued fractional Brownian--motion. For brevity we suppose that $\omega$ can be presented by
\begin{equation*}
  \omega(t)=\sum_iq_i^\frac12\omega_i(t)e_i
\end{equation*}
where $(\omega_i)_{i\in \mathbb N}$ is a sequence of one-dimensional independent $\beta$--H{\"o}lder continuous standard fractional Brownian--motions for any $\beta<H$ such that
 $\sum_iq_i<\infty$. Then $\omega$ can be interpreted as a $\beta$--H{\"o}lder--continuous fractional Brownian--motion in $V$ with Hurst--parameter $H\in (0,1)$.
 The covariance is given by the operator $Q$ with is a diagonal operator in our standard basis with diagonal elements $q_i$.
 Let us also denote by $\omega^n,\,\omega_i^n$ the piecewise linear approximations with respect to the equidistant partition  $\{t_i^n\}$ of $[0,T]$ of length $2^{-n}T=\delta$.\\

Throughout this section we assume that  $\omega$ has a  Hurst--parameter $H\in (1/3,1/2]$. Since $\omega^n$ is smooth we can define $(\omega^n\otimes_S\omega^n)(s,t)$ as a Bochner--integral with respect to the Lebesgue--measure:
\begin{equation*}
  E(\omega^n\otimes_S\omega^n)(s,t)=\int_s^t\int_s^\xi S(\xi-r)Ed\omega^n(r)\otimes_V d\omega^n(\xi)
\end{equation*}
for $E\in L_2(V,\hat V)$.
By an integration by parts argument this integral can be rewritten as
\begin{align*}
\begin{split}
    \int_s^t&E(\omega^n(\xi)-\omega^n (s))\otimes_V  d\omega^n(\xi)+\int_s^tA\int_s^\xi S(\xi-r)E(\omega^n(r)-\omega^n(s))dr\otimes_V d\omega^n(\xi).
\end{split}
\end{align*}
This motivates to interpret $(\omega\otimes_S\omega)(s,t)$ as
\begin{equation}\label{lastf}
  \int_s^tE(\omega(\xi)-\omega (s))\otimes_V  d\omega(\xi)
    +\int_s^tA\int_s^\xi S(\xi-r)E(\omega(r)-\omega(s))dr\otimes_V d\omega(\xi)
\end{equation}
where we have to give an appropriate meaning to both integrals. We start with the first one which is abbreviated in the following by $ E(\omega\otimes\omega)(s,t)$.

\begin{lemma}\label{stel2}
Let $\hat V=V_\kappa$ for $\kappa>0$ such that $\sum_i\lambda_i^{-2\kappa}<\infty$.
The sequence $((\omega^n\otimes\omega^n))_{n\in\NN}$,
which elements can be represented component--wise by
\begin{equation*}
q_j^\frac12 q_k^\frac12\int_s^t(\omega_j^n(\xi)-\omega_j^n(s)) d\omega_k^n(\xi)
\end{equation*}
converges on a set of full measure  in $C_{2\beta} (\bar\Delta_{0,T};L_2(L_2(V,V_\kappa),V\otimes V))$ to $(\omega\otimes\omega)$.
\end{lemma}

\begin{proof}
Consider the orthonormal basis $(E_{ij})_{i,j\in\NN}$ of $L_2(V,V_\kappa)$ given  by
\begin{equation*}
E_{ij}e_k=\left\{\begin{array}{lcl}
    0&:& j\not= k\\
   \frac{e_i}{\lambda_i^\kappa}&:& j= k,
    \end{array}
    \right.
\end{equation*}
and $(e_l\otimes_V e_k)_{l,k\in\NN}$, the orthonormal basis of $V\otimes V$.
First note that, for $(s,t)\in \bar\Delta_{0,T}$,
\begin{align*}
\|(\omega^n \otimes \omega^n)(s,t)\|^2_{L_2(L_2(V,V_\kappa),V\otimes V)}&= \sum_{i,j}\sum_{l,k}  (E_{ij}( \omega^n \otimes \omega^n )(s,t), e_l\otimes_V e_k)_{V\otimes V}^2 \\
& \leq
  \sum_i\lambda_i^{-2\kappa}\sum_{i,j}q_jq_k \bigg(\int_s^t(\omega_j^n(\xi)-\omega_j^n(s))d\omega_k^n(\xi)\bigg)^2,
\end{align*}
and hence we have to study the behavior of the last sum. Let us denote
\begin{align*}
 A_{j,k}^n(s,t)\footnotemark:&=   \int_s^t(\omega_j(\xi)-\omega_j(s))d\omega_k(\xi)-\int_s^t(\omega_j^n(\xi)-\omega_j^n(s))d\omega_k^n(\xi).
\end{align*}
By symmetry we can assume $j\leq k$. In fact we assume that $j<k$ since the case $j=k$ is easier, see a comment at the end of the proof.
\footnotetext{In \cite{DeNeTi10}, $A_{j,k}^n(s,t)$ are considered over the square $[0,T]^2$. However, since they are symmetric w.r.t. the diagonal of $[0,T]$ it is sufficient to consider these elements over $\bar\Delta_{0,T}$.}

To estimate the continuous element $A_{j,k}^n(s,t)$ we apply the Lemma 3.7 in Deya {\it et al.} \cite{DeNeTi10}, which claims that for $p\geq 1$ there exists $K_{\beta, p}$ such that
\begin{align}\label{tail}
\|A_{j,k}^n\|_{2\beta} \leq K_{\beta, p} (R_{n,p}^{j,k}+ \|\omega_j-\omega_j^n\|_{\beta} \|\omega_k\|_{\beta}+\|\omega_k-\omega_k^n\|_{\beta} \|\omega_j^n\|_{\beta}),
\end{align}
where
\begin{align*}
R_{n,p}^{j,k}:=\bigg(\int_0^T\int_0^T\frac{|A_{j,k}^n(s,t)|^{2p}}{|t-s|^{4\beta p+2}}dsdt\bigg)^{1/(2p)}.
\end{align*}
In particular, from the
proof of Lemma 3.7 in \cite{DeNeTi10} we know that
\begin{equation*}
   \EE (R_{n,p}^{j,k})^{2p}\le c n^{-4p(H-\beta^{\prime})}<\infty,
\end{equation*}
for $\beta<\beta^{\prime}<H$, being $\beta^{\prime}$ close enough to $H$ and $p$ large enough. Indeed, let us take $p$ large enough such that $4p(H-\beta^{\prime})>1$, and thus
\begin{align*}
    \PP(\sum_{j,k}q_j q_k(R_{p,n}^{j,k})^2>o(n)^2)\le \frac{({\rm tr}_V Q)^{2(p-1)}}{o(n)^{2p}}\sum_{j,k}q_j q_k \EE(R_{n,p}^{j,k})^{2p}
    \le \frac{c n^{-4p(H-\beta^{\prime})}} {o(n)^{2p}}.
\end{align*}
For an appropriate sequence $(o(n))_{n\in\NN}$ with limit zero, the right-hand side has a finite sum. Then by the Borel--Cantelli--lemma, $(\sum_{j,k}q_jq_k(R_{p,n}^{j,k})^2)_{n\in\NN}$ tends to zero almost surely. In a similar manner we obtain the convergence of the last terms in \eqref{tail}. It suffices to take into account that, for $\beta< \beta^{\prime}<H$,
\begin{align}\label{eq20}
\begin{split}
    \|\omega_j-\omega_j^n\|_{\beta}\le  G_{\beta^{\prime}}(j,\omega)n^{\beta-\beta^{\prime}},\;
    \|\omega_j\|_{\beta}\le G_{\beta^{\prime}}(j,\omega),\;
    \|\omega_j^n\|_{\beta}\le G_{\beta^{\prime}}(j,\omega)
    \end{split}
    \end{align}
where $G_{\beta^{\prime}}(j,\omega)\ge   \|\omega_j\|_{\beta^{\prime}}$ and  $G_{\beta^{\prime}}(j,\omega) \in L_{p}(\Omega)$  for any  $p\in\NN$ are iid random variables, see Kunita \cite{Kunita90} Theorem  1.4.1. We then have
\begin{align*}
   & \PP(\sum_{j,k}q_j q_k\|\omega_j^n-\omega_j\|_{\beta}^2\|\omega_k\|_{\beta}^2>o(n)^2)\\
    \le & \frac{({\rm tr}_VQ)^{2(p-1)}}{o(n)^{2p}} \sum_{j,k}q_j q_k (\EE G_{\beta^{\prime}}(j,\omega)^{4p})^\frac12(\EE G_{\beta^{\prime}}(k,\omega)^{4p})^\frac12n^{2p(\beta-\beta^{\prime})}\le \frac{c n^{2p(\beta-\beta^{\prime})}}{o(n)^{2p}}.
\end{align*}
For $p$ chosen sufficiently large and an appropriate zero--sequence $(o(n))_{n\in\NN}$ we obtain the almost sure convergence of
$(\sum_{j,k}q_j q_k\|\omega_j^n-\omega_j\|_{\beta}^2 \|\omega_k\|_{\beta}^2)_{n\in \NN}$. Similarly we can treat the last term  of \eqref{tail}, that is,
$(\sum_{j,k}q_j q_k\|\omega_j^n-\omega_j\|_{\beta}^2 \|\omega_k^n\|_{\beta}^2)_{n\in \NN}$.
Finally,
\begin{align*}
 A_{j,j}^n(s,t)&=   \frac{1}{2}(\omega_j(t)-\omega_j(s))^2- \frac{1}{2}(\omega_j^n(t)-\omega_j^n(s))^2
 \end{align*}
 and thanks to \eqref{eq20} $\| A_{j,j}^n\|_{2\beta}\leq G_{\beta^{\prime}}(j,\omega)^2n^{2(\beta-\beta^{\prime})}$, which completes the proof.
\end{proof}

\begin{lemma}\label{stel3}
Suppose that there exists $\gamma$ such that $\gamma+\beta>1$ and
\begin{equation*}
\sum_{i}\lambda_i^{2\gamma-2\kappa}<\infty.
\end{equation*}
Then the mapping
\begin{align*}
  ((s,t),E) & \in \bar\Delta_{0,T}\times L_2(V,V_\kappa)\mapsto \int_s^t\bigg(A\int_s^\xi S(\xi-r)E(\omega(r)-\omega(s))dr\bigg)\otimes_Vd\omega(\xi)
\end{align*}
is in $C_{2\beta}(\bar\Delta_{0,T};L_2(L_2(V,V_\kappa),V\otimes V))$.
\end{lemma}
\begin{proof}
Thanks to Pazy \cite{Pazy} Theorem 4.3.5 (iii),
\begin{align}\label{neu51}
\begin{split}
  \bigg|A&\int_s^\xi S(\xi-r)E_{ij}(\omega(r)-\omega(s))dr\bigg|
=\frac{q_j^\frac12}{\lambda_i^{\kappa}}\bigg|\int_s^\xi AS(\xi-r)e_i(\omega_j(r)-\omega_j(s))dr\bigg|\\
  &\le c\frac{q_j^\frac12}{\lambda_i^{\kappa}}\|\omega_j\|_{\beta}(\xi-s)^{\beta},
\end{split}
\end{align}
and applying Bensoussan and Frehse \cite{BenFre00} Corollary 2.1 we also have
\begin{align}\label{neu52}
\begin{split}
\bigg|A&\int_s^\xi S(\xi-r)E_{ij}(\omega(r)-\omega(s))dr-A\int_s^{\xi^\prime} S(\xi^\prime-r)E_{ij}(\omega(r)-\omega(s))dr\bigg|\\
  \le & \frac{q_j^\frac12}{\lambda_i^{\kappa-\gamma}}\bigg|\int_s^\xi (-A)^{1-\gamma}S(\xi-r)e_i(\omega_j(r)-\omega_j(s))dr\\
  &\qquad -\int_s^{\xi^\prime} (-A)^{1-\gamma}S(\xi^\prime-r)e_i(\omega_j(r)-\omega_j(s))dr\bigg|\le c\frac{q_j^\frac12}{\lambda_i^{\kappa-\gamma}}\|\omega_j\|_{\beta}|\xi-\xi^\prime|^\gamma.
\end{split}
\end{align}
Therefore, for an $\alpha < \gamma$ such that $\beta >1-\alpha $, we can define the integral
\begin{align*}
& \int_s^t\bigg(A\int_s^\xi S(\xi-r)e_i(\omega_j(r)-\omega_j(s))dr\bigg) d\omega_k(\xi)\\
=&(-1)^\alpha \int_s^tD_{s+}^\alpha\bigg( A \int_s^\cdot S(\cdot-r)e_i(\omega_j(r)-\omega_j(s))dr\bigg)[\xi]D_{t-}^{1-\alpha}(\omega_k)_{t-}[\xi]d\xi.
\end{align*}
For the definition of the so-called fractional derivatives $D_{s+}^\alpha$ and $D_{t-}^{1-\alpha}$ and the definition of the stochastic integral in terms of these expressions we refer to \cite{Zah98}. Note that as a consequence of \eqref{neu51} and \eqref{neu52}, and because $\gamma >\beta$,
\begin{align*}
\bigg|D_{s+}^\alpha\bigg(A \int_s^\cdot S(\cdot-r)e_i(\omega_j(r)-\omega_j(s))dr\bigg)[ \xi]\bigg| &\le c\|\omega_j\|_{\beta}
  (\xi-s)^{\beta-\alpha}.
  \end{align*}
Since $\beta+\alpha>1$, $|D_{t-}^{1-\alpha}(\omega_k)_{t-}[\xi]|\le c\|\omega_k\|_{\beta}(t-\xi)^{\alpha+\beta-1}$, and then
\begin{align*}
  \bigg( & \sum_{i,j,k} \bigg( \int_s^t\bigg(\frac{q_j^\frac12}{\lambda_i^{\kappa}}\int_s^\xi AS(\xi-r)e_i(\omega_j(r)-\omega_j(s))dr\bigg)\otimes_V q_k^\frac12d\omega_k(\xi )\bigg)^2\bigg)^\frac12
  \\
  &\le \bigg(\sum_i\lambda^{2\gamma-2\kappa}_i\bigg)^\frac12 \bigg(\sum_{j,k}q_jq_k\|\omega_j\|_{\beta}^2\|\omega_k\|_{\beta}^2\bigg)^\frac12\int_s^t (\xi-s)^{\beta-\alpha}(t-\xi)^{\alpha+\beta-1} d\xi,
\end{align*}
see \cite{GLSch-local} to find the estimate of the integral in terms of the norms of the fractional derivatives. Now it suffices to take into account that the last integral can be estimated by $c(t-s)^{2\beta}$, which follows from the definition of the Beta function.
\end{proof}
\begin{remark}\label{new}
Replacing in the above proof $\omega$ by $\omega^n-\omega$ we obtain
\begin{align*}
&\lim_{n\to\infty}\int_s^tA\int_s^\xi S(\xi-r)E(\omega^n(r)-\omega^n(s))dr\otimes_V d\omega^n(\xi)\\
=&\int_s^tA\int_s^\xi S(\xi-r)E(\omega(r)-\omega(s))dr\otimes_V d\omega(\xi)
\end{align*}
in $C_{2\beta}(\bar\Delta_{0,T};L_2(L_2(V,V_\kappa),V\otimes V))$.
Indeed, in the proof of Lemma \ref{stel2} we have shown
\begin{equation*}
  \lim_{n\to\infty}\sum_{i,j}q_i q_i(\|\omega_i-\omega_i^n\|_{\beta}^2\|\omega_j\|_{\beta}^2+\|\omega_j-\omega_j^n\|_{\beta}^2\|\omega_i^n\|_{\beta}^2)=0\quad \text{a.s.}\\
\end{equation*}
\end{remark}

In view of Lemma \ref{stel2} and Remark \ref{new}, we conclude with the proof of Theorem \ref{stet1}.

\begin{corollary}\label{stec1}
The mapping $\bar\Delta_{0,T}\ni (s,t)\to (\omega\otimes_S\omega)(s,t)$ is continuous.
\end{corollary}

Indeed, from the above estimate we have the convergence
\begin{equation*}
   \lim_{n\to\infty}\sup_{(s,t)\in\bar\Delta_{0,T}}\|(\omega\otimes_S\omega)(s,t)-(\omega^n\otimes_S\omega^n)(s,t)\|_{L_2(L_2(V,V_\kappa),V\otimes V)}=0
\end{equation*}
and straightforwardly
\begin{equation*}
   \bar\Delta_{0,T}\ni (s,t)\to (\omega^n\otimes_S\omega^n)(s,t)
\end{equation*}
is continuous.

\smallskip

\section{The construction of $(\omega\otimes_S\omega)$ for a  Brownian--motion}\label{stes3}

We consider $\omega$ to be a trace class Brownian--motion with a positive symmetric trace--class covariance operator $Q$. Therefore, in that case $\omega$ is a trace class fractional Brownian--motion
with Hurst--parameter $H=1/2$. We want to formulate a weaker condition for the existence of $(\omega\otimes_S\omega)$
than the assumption of Lemma \ref{stel3}. Again we will take $\hat V=V_\kappa$ and will determine which conditions $\kappa$ must satisfy. \\

Now let us split $(\omega\otimes_S\omega)$ into its components.
Taking into account that $e^{-\lambda_it}$ gives the decomposition of $S(t)$ with respect to the base $(e_i)_{i\in\NN}$, we characterize  $(\omega\otimes_S\omega)$ for $l=i$ by
\begin{align}\label{localeq24}
\begin{split}
 (E_{ij}&(\omega\otimes_S\omega)(s,t),e_l\otimes_Ve_k)_{V\otimes V}=\frac{q_j^\frac12q_k^\frac12}{\lambda_i^\kappa}\int_s^t\int_s^\xi e^{-\lambda_i(\xi-r)}d\omega_j(r)\circ d\omega_k(\xi)\\
  =& \frac{q_j^\frac12q_k^\frac12}{\lambda_i^{\kappa}}\bigg(\int_s^t(\omega_j(\xi)-\omega_j(s))\circ d\omega_k(\xi)\\
  &\qquad -
  \lambda_i\int_s^t\int_s^\xi e^{-\lambda_i(\xi-r)}(\omega_j(r)-\omega_j(s))dr d\omega_k(\xi)\bigg)\\
  = &\frac{q_j^\frac12q_k^\frac12}{\lambda_i^\kappa}\int_s^t\int_s^\xi e^{-\lambda_i(\xi-r)}d\omega_j(r) d\omega_k(\xi)-\frac12\delta_{jk}\frac{q_j^\frac12q_k^\frac12}{\lambda_i^\kappa}(t-s)
\end{split}
\end{align}
and for $l\not=i$ by 0.
Here $\circ$ means Stratonovich-integration where for the inner stochastic integration Ito-- and Stratonovich--integrals are the same.
On the right hand side we have Ito--integration, and the last term there is the Ito--correction which only appears for $j=k$, expressed by the Kronnecker--symbol $\delta_{jk}$.
Note that the second line in (\ref{localeq24}) is obtained by stochastic integration by parts. Let us abbreviate
\begin{equation*}
  \int_s^t\int_s^\xi e^{-\lambda_i(\xi-r)}d\omega_j(r)\circ d\omega_k(\xi)=:(\omega\otimes_S\omega)_{ijk}.
\end{equation*}
In this section, $\omega^n,\,\omega_i^n$ are piecewise linear approximations of the Brownian motions $\omega, \,\omega_i$ with respect to the equidistant partition  $\{t_i^n\}$ of $[0,T]$ of length $2^{-n}T=\delta$.

We now deal with the convergence of $(\omega^n\otimes_S\omega^n)_{ijk}$ to $(\omega\otimes_S\omega)_{ijk}$.
At first we present some lemmata
which will be needed for this purpose.

\smallskip

As a preparatory result for the following we need (see Karatzas and Shreve \cite{KarShr91}, Exercise 3.25, Page 163):

\begin{lemma} \label{stel4} Let $\omega$ denote an one-dimensional Brownian motion and $x$ be a measurable, adapted process satisfying
\begin{align*}
\EE \int_0^T |x(\xi)|^{2p} d\xi < \infty
\end{align*}
for some real numbers $T > 0$ and $p\geq 1$, then
\begin{align*}
\EE \bigg(\int_0^T x(\xi)d\omega(\xi)\bigg)^{2p} \leq (p(2p-1))^p T^{p-1} \EE \int_0^T |x(\xi)|^{2p} d\xi.
\end{align*}
\end{lemma}
\begin{lemma}\label{localnew}
For any $p\in\NN$ there exists a $c_p>0$ such that for any $M\in \NN$

\begin{equation*}
  \sum_{k_1+\cdots+k_M=p}\frac{(2p)!}{(2k_1)!(2k_2)!\cdots(2k_M)!}
  \le c_pM^p.
\end{equation*}
Let $(x_i)_{i=1,\cdots M}$ a sequence of independent random variables in $L_{2p}$ where the odd moments are zero.
If in addition $\EE x_i^{2p}\leq c $, $\EE x_i^{0}=1$ for $x_i\not\equiv 0$ then
$$\EE\bigg(\sum_{i=1}^M x_i\bigg)^{2p} \leq c_p c M^p,$$
assuming that all terms containing at least one odd power disappear.
\end{lemma}
\begin{proof}
We have
\begin{equation*}
   \sum_{k_1+\cdots+k_M=p}1\leq M^p.
\end{equation*}
On the other hand
\begin{equation*}
\sup_{k_1+\cdots+k_M=p}\frac{(2p)!}{(2k_1)!(2k_2)!\cdots(2k_M)!}\le c_p\le (2p)!.
\end{equation*}
Hence a bound for this expression can be chosen independently of $M$.

To conclude the proof it suffices to apply the multinomial theorem, which reduces to the following situation since the terms containing odd powers are neglected:
\begin{align*}
\EE\bigg(\sum_{i=1}^M x_i\bigg)^{2p} &=\sum_{a_1+\cdots+a_M=2p}  \binom{2p}{a_1,\cdots,a_M}\EE x_1^{a_1} \cdots \EE x_M^{a_M}\\
&=\sum_{k_1+\cdots+k_M=p}  \binom{2p}{2k_1,\cdots,2k_M}\EE x_1^{2k_1} \cdots \EE x_M^{2k_M}\\
&\le \sum_{k_1+\cdots+k_M=p}  \binom{2p}{2k_1,\cdots,2k_M}\prod_{m=1,\cdots, M,k_m>0}(\EE x_m^{2p})^\frac{2k_m}{2p}\\
 &\leq c_p c M^p.
\end{align*}
If some of the $k_m=0$ then $\EE x_m^0=1$ and the corresponding terms would be removed from the above product.
\end{proof}

\begin{lemma}\label{stel5}
For every $T>0$, $p\in\NN$ and a sufficiently small $\eps>0$ there exists  $c_p>0$  such that for $(s,t)\in\bar\Delta_{0,T}$, $i,j,k,n\in\NN$
\begin{align}\label{localeq27}
\begin{split}
 & \EE\big|(\omega\otimes_S\omega)_{ijk}-(\omega^n\otimes_S\omega^n)_{ijk} \big|^{2p}\le c_p \delta^\eps\lambda_i(t-s)^{2p-\eps}.
  \end{split}
\end{align}
\end{lemma}
\begin{proof}
We here only study the case $j=k$ since the case $j\not=k$ can be studied similarly, see also Friz and Hairer  \cite{FriHai14} Page 33f..
For the following we assume that $\lambda_1\ge 1$ without lost of generality. For a partition interval
$[t_{m-1}^n,t_m^n)$ we use the notation
\begin{equation*}
\Delta_k^{\delta}(m)=\omega_k(t_m^n)-\omega_k(t_{m-1}^n)
\end{equation*}
and when $s<t\in [t_{m-1}^n,t_m^n)$ we will also use the notation
\begin{equation*}
  \Delta_k^{t-s}(m)=\omega_k(t)-\omega_k(s).
\end{equation*}
We divide the proof in several cases:

(i) We consider at first the case that $s,\,t\in [t_{m-1}^n,t_m^n]$. In that situation, the difference of double integrals we want to estimate is given for any $i$ by
\begin{align}\label{localeq22}
\begin{split}
& \frac{\Delta_k^\delta(m)^2}{\delta^2\lambda_i^2}(e^{-\lambda_i(t-s)}+\lambda_i(t-s)-1)
  -\frac{\Delta^{t-s}_k(m)^2}{2}\\
  &\quad +\lambda_i\int_s^t\int_s^\xi e^{-\lambda_i(\xi-r)}(\omega_k(r)-\omega_k({s}))dr d\omega_k(\xi)
  \end{split}
\end{align}
which follows by the second part of \eqref{localeq24}, and where the first expression corresponds to the integral with respect to the piecewise linear
approximated  Brownian--motions.
The first two expressions of \eqref{localeq22} can be estimated by
\begin{align}\label{localeq26}
\begin{split}
 \bigg|\frac{\Delta_k^\delta(m)^2}{\delta^2\lambda_i^2}&(e^{-\lambda_i(t-s)}+\lambda_i(t-s)-1)-\frac12\Delta_k^{t-s}(m)^2\bigg|\\
 &\le \Delta_k^{t-s}(m)^2\bigg|\frac12-\frac{e^{-\lambda_i(t-s)}+\lambda_i(t-s)-1}{(t-s)^2\lambda_i^2}\bigg|\\
 &+\bigg(\frac{\Delta_k^\delta(m)^2(t-s)^2}{\delta^2}-\Delta_k^{t-s}(m)^2\bigg)\frac{e^{-\lambda_i(t-s)}+\lambda_i(t-s)-1}{(t-s)^2\lambda_i^2}.
\end{split}
\end{align}
Note that the function
\begin{equation*}
 \RR^+\ni x\mapsto \frac12-\frac{e^{-x}+x-1}{x^2}=O(x)\quad\text{for }x\to 0^+
\end{equation*}
defined by  zero at zero, is increasing and bounded by $1/2$, having derivatives bounded by 1 for $x>0$ and
\begin{equation*}
 \RR^+\ni x\mapsto \frac{e^{-x}+x-1}{x^2}\in [0,\frac12]
\end{equation*}
given by $1/2$ at zero. Hence
\begin{equation*}
  \bigg(\frac12 -\frac{e^{-\lambda_i(t-s)}+\lambda_i(t-s)-1}{(t-s)^2\lambda_i^2}\bigg)^{2p}\le \frac12-\frac{e^{-\lambda_i(t-s)}+\lambda_i(t-s)-1}{(t-s)^2\lambda_i^2}\le \lambda_i(t-s).
\end{equation*}
Since $t-s<\delta$, the $2p$-moment of the first expression on the right hand side of \eqref{localeq26} can be estimated by $\tilde c_p^1(t-s)^{2p}\lambda_i\delta$ and the second one by $\hat c_p^1(t-s)^{2p-\eps}\delta^\eps$ for a sufficiently small $\eps$
and for appropriate constants $\tilde c_p^1,\,\hat c_p^1$. Considering the last expression of \eqref{localeq22}, by H{\"o}lder's inequality for the inner integral we obtain
\begin{align*}
\begin{split}
  \EE&\bigg(\int_{s}^\xi e^{-\lambda_i(\xi-r)}(\omega_k(r)-\omega_k(s))dr\bigg)^{2p}\\
  &\le\bigg(\int_{s}^\xi e^{\frac{-2p}{2p-1}\lambda_i(\xi-r)}dr\bigg)^{2p-1}
  \int_{s}^\xi \EE(\omega_k(r)-\omega_k(s))^{2p}dr\\
  &\le \bar c^{1}_p\bigg(\frac{2p-1}{2p}\bigg)^{2p-1}\frac{1}{\lambda_i^{2p-1}}(1-e^{\frac{-2p}{2p-1}\lambda_i(\xi-s)})^{2p-1}(\xi-s)^{p+1}.
\end{split}
\end{align*}
Applying Lemma \ref{stel4} we get
\begin{equation*}
 \EE\bigg(\lambda_i\int_s^t\int_{s}^\xi e^{-\lambda_i(\xi-r)}(\omega_k(r)-\omega_k(s))dr d\omega_k(\xi)\bigg)^{2p}\le  \tilde c_p^1(t-s)^{2p}\delta\lambda_i.
\end{equation*}
Note that all the constants depending on $p$ that have appeared can be chosen independently of $i,\,k$.\\

(ii) If $ s,\,t$ are in two neighbored intervals, say, for instance, $t_{m-1}^n\le s< t_m^n\le t \le t_{m+1}^n$, we can estimate the expression in \eqref{localeq27} in a similar way than before, just dividing the region of integration into two triangles $\bar\Delta_{s,t_{m}^n}$, $\bar\Delta_{t_m^n,t}$ and a rectangle $[t_m^n,t]\times [s ,t_m^n]$. The integrals with respect to the triangles can be estimated as in the step (i) while the estimates with respect to the rectangle are considered  below.\\

(iii) Let us now assume in general $t-s> \delta$. In addition suppose that $0\le t_{m_0-1}^n\le s<t_{m_0}^n<$ $t_{m_1}^n< t\le t_{m_1+1}^n\le T$. We begin considering the integrals, denoted by $I_m^{\Delta,\delta,i,k,k}$, of \eqref{localeq27} for $j=k$  and over any of the triangles $\bar\Delta_{t_{m}^n,t_{m+1}^n}$ along the hypothenuse of $\bar\Delta_{0,T}$.  Then following the step (i) of the proof, since in the same triangle $\bar\Delta_{t_{m}^n,t_{m+1}^n}$ the distance is just given by $\delta$, the corresponding estimate of  \eqref{localeq27} is bounded by $\tilde c_p^1\lambda_i\delta^{2p+1}$ (the second term of (\ref{localeq26}) cancels out in this case).

To estimate the $2p$--moment of the sum of all these integrals we apply H{\"o}lder's inequality and Lemma \ref{localnew}. Since the number of these rectangles is of order $(t-s)/\delta$, we get
\begin{equation*}
  \EE\bigg(\sum_m I_{m}^{\Delta,\delta,i,k,k}\bigg)^{2p}=O\bigg(\bigg(\frac{t-s}{\delta}\bigg)^{2p-1}
  \frac{t-s}{\delta}\bigg)\delta^{2p+1}\lambda_i\quad\text{for }\delta\to 0^+.
\end{equation*}
Hence taking also the triangles $\bar\Delta_{s,t_{m_0}^n}$, $\bar\Delta_{t_{m_{1}}^n,t}$ into account there exists an $c_p$ such that for any $\delta\ge 1$ and $t_{m_0}-s\le \delta$ and $t-t_{m_1}\leq \delta$,
\begin{equation*}
   \EE\bigg(\sum_m I_{m}^{\Delta,\delta,i,k,k}\bigg)^{2p}\le c_p\lambda_i(t-s)^{2p-\eps}\delta^\eps.\\
\end{equation*}

(iv)
Consider now the integrals over one of the rectangles $[t_{m^\prime-1}^n,t_{m^\prime}^n]\times[t_{m}^n,t_{m+1}^n]$  such that $m\ge m^\prime$.
First we note that
\begin{align*}
  &I_{m^\prime-1,m}^{R,\delta,i,k,k}:=\int_{m\delta}^{(m+1)\delta}\int_{(m^\prime-1)\delta}^{m^\prime\delta}e^{-\lambda_i(\xi-r)}d\omega_k(r)d\omega_k(\xi)\\
 & \qquad  \qquad  -\int_{m\delta}^{(m+1)\delta}\int_{(m^\prime-1)\delta}^{m^\prime\delta}e^{-\lambda_i(\xi-r)}d\omega_k^n(r)d\omega^n_k(\xi)\\
  &=_{distr}e^{-\lambda_i\delta(m-m^\prime)}
  \int_{m\delta}^{(m+1)\delta}\int_{(m-1)\delta}^{m\delta}\big(e^{-\lambda_i(\xi-r)}-\frac{(1-e^{-\lambda_i\delta})^2}{\delta^2\lambda_i^2}\big)d\omega_k(r)d\omega_k(\xi)\\
  &=e^{-\lambda_i\delta(m-m^\prime)}I_{m-1,m}^{R,\delta,i,k,k}
\end{align*}
where the factor in front of the integral is less than 1.
Considering the $2p$--moment of these integrals, thanks to Lemma \ref{stel4} and H\"older's inequality, we have
\begin{align*}
  \EE\big(I_{m-1, m}^{R,\delta,i,k,k}\big)^{2p}&\le  c_p \delta\lambda_i\frac{1}{(\delta\lambda_i)}\delta^{2p-2}\int_{m\delta}^{(m+1)\delta}\int_{(m-1)\delta}^{m\delta}
  \big(e^{-\lambda_i(\xi-r)}-\frac{(1-e^{-\lambda_i\delta})^2}{\delta^2\lambda_i^2}\big)^{2p}drd\xi\\
  &\le c_p\delta\lambda_i\delta^{2p}\int_{0}^1\int_{-1}^0
  \big(e^{-\lambda_i\delta(\xi-r)}-\frac{(1-e^{-\lambda_i\delta})^2}{\delta^2\lambda_i^2}\big)^{2p}/(\delta\lambda_i) drd\xi\\
  &\le c_p\delta\lambda_i\delta^{2p}\int_{0}^1\int_{-1}^0\sup_{x>0,y\in [0,2]}
  \big(e^{-xy}-\frac{(1-e^{-x})^2}{x^2}\big)^{2p}/x drd\xi.
\end{align*}
The above supremum is finite. Hence it is easily seen that this integrand is bounded independently of $y$ for $x\to+\infty$.  Consider now the integral in \eqref{localeq27} over the union of squares inside $\bar\Delta_{0,T}$. Expanding
\begin{equation*}
  \EE(\sum_{m^\prime\le m}I_{m^\prime, m}^{R,\delta,i,k,k})^{2p}
\end{equation*}
we can cancel all terms containing an odd power by the independence of all these integrals, such that the number of terms in this multinomial is of order
$((t-s)/\delta)^{2p}$  by Lemma \ref{localnew}. Hence
\begin{equation*}
  \EE(\sum_{m^\prime\le m}I_{m-1, m}^{R,\delta,i,k,k})^{2p}=O\bigg(\frac{(t-s)^{2p}}{\delta^{2p}}\bigg)\lambda_i\delta^{2p+1}=\lambda_iO((t-s)^{2p})\delta,\quad\delta\to 0^+
\end{equation*}
for $\delta\to 0^+$. Therefore, for any $p,\,i,\,j,\,k$ there exists a $c_p$ such that
\begin{equation*}
  \EE(\sum_{m^\prime\le m}I_{m-1, m}^{R,\delta,i,k,k})^{2p}\le c_p\lambda_i\delta(t-s)^{2p}.
\end{equation*}

(v) In a similar manner we can consider the rectangles along the cathetus  of $\bar\Delta_{0,T}$. We omit the calculations.

\end{proof}

\begin{lemma}\label{locall110}
Suppose that for a $\nu>0$ and $p\in\NN$ we have
\begin{equation*}
  \sum_i\lambda_i^{\frac{-\nu p}{p-1}}<\infty \quad\text{and  }\sum_i\lambda_i^{p(\nu-2\kappa)+1}<\infty.
\end{equation*}
Then
there exists a $c_p>0$ such that for all $(s,t)\in \bar\Delta_{0,T}$
\begin{align*}
\EE\|(\omega\otimes_S\omega)(s,t)\|_{L_2(L_2(V,V_\kappa),V\otimes V)}^{2p}& \le c_p(t-s)^{2p}.
\end{align*}
In addition the random field $(\omega \otimes_S\omega)$ has a continuous version.
\end{lemma}
\begin{proof} In the proof, $c_p$ denotes a constant that can vary from one line to other.
First of all,
\begin{align*}
&\EE(\omega\otimes_S\omega)_{ijk}^{2p}\le c_p(t-s)^{2p}+c_p \EE\bigg(\int_s^t\int_s^\xi e^{-\lambda_i(\xi-r)}d\omega_j(r) d\omega_k(\xi)\bigg)^{2p}.
\end{align*}
To estimate this last expectation, note that
\begin{equation*}
  \EE\bigg(\int_s^\xi e^{-\lambda_i(\xi-r)}d\omega_j(r)\bigg)^{2p} \le c_p \bigg(\EE\bigg(\int_s^\xi e^{-\lambda_i(\xi-r)}d\omega_j(r)\bigg)^2\bigg)^p\le c_p(\xi-s)^{p},
\end{equation*}
due to the fact that the above integral is a Gau\ss--variable. Applying Lemma \ref{stel4} we obtain that $\EE(\omega\otimes_S\omega)_{ijk}^{2p}\le c_p (t-s)^{2p}$. Now, by H{\"o}lder's--inequality we have
\begin{align*}\begin{split}
  &\EE\|(\omega\otimes_S\omega)(s,t)\|_{L_2(L_2(V,V_\kappa),V\otimes V))}^{2p} = \EE\bigg( \sum_{i,j,k}q_jq_k\lambda_i^{-2\kappa}(\omega\otimes_S \omega)_{ijk}^2\bigg)^{p} \\
  & \le \bigg(\sum_{ijk}q_jq_k{ \lambda_i^{\frac{-\nu p}{p-1}}}\bigg)^{p-1}\sum_{ijk}q_jq_k\lambda_i^{{p\nu}-2p\kappa}
  \EE(\omega\otimes_S \omega)_{ijk}^{2p}\\
  &\le \bigg(\sum_{ijk}q_j  q_k{\lambda_i^{\frac{-\nu p}{p-1}}}\bigg)^{p-1}\sum_{ijk}q_jq_k\lambda_i^{{p\nu}-2p\kappa}c_p(t-s)^{2p}.
\end{split}
\end{align*}

Next we would like to sketch that the random field $(\omega\otimes_S\omega)$ has a continuous version in $L_2(L_2(V,V_\kappa),V\otimes V)$. Let us first consider the Ito--version of the
Hilbert--space valued stochastic integral $(\omega\otimes_S\omega)(0,\cdot)$ given by
\begin{equation}\label{steeq101}
  t\mapsto E(\omega\otimes_S\omega)(0,\cdot)+\frac12\sum_i (e_i,Ee_i)(e_i\otimes_V e_i)t.
\end{equation}
To see that such an Ito--Integral makes sense we consider at first a predictable stochastic process  $G:[0,T]\mapsto V$.
We define
\begin{equation*}
 w\mapsto (G(\xi)\otimes)w=G(\xi)\otimes_V w\in V\otimes V,\quad w\in V.
\end{equation*}
If $G\otimes$ satisfies the condition
\begin{align*}
\EE\bigg(&\int_0^t(G(\xi)\otimes)d\omega(\xi)\bigg)^2=
  \int_0^t\sum_iq_i \EE \|(G(\xi)\otimes)e_i\|_{V\otimes V}^2d\xi\\
  &=\EE\int_0^t\sum_{ijk}q_i(e_j\otimes_V e_k,G(\xi)\otimes_Ve_i)_{V\otimes V}^2d\xi\\&=\sum_{ij}\int_0^tq_i\EE(e_j,G(\xi))^2d\xi=\sum_iq_i\int_0^T\EE\|G(\xi)\|^2d\xi<\infty
\end{align*}
then the Ito--integral
\begin{equation*}
  \int_0^t(G(\xi)\otimes)d\omega(\xi)\in V\otimes V
\end{equation*}
is well defined.
We apply this formula to define an $L_2(L_2(V,V_\kappa),V\otimes V)$-valued Ito--integral. In particular we set
\begin{equation*}
  G(\xi)=G_{ij}(\xi)=\int_0^\xi S(\xi-r)E_{ij}d\omega(r)
\end{equation*}
where $(E_{ij})_{i,j\in\NN}$ is the complete orthonormal system of $L_2(V,V_\kappa)$. Then $(\omega\otimes_S\omega)(s,t)$ is well defined on
$L_2(L_2(V,V_\kappa),V\otimes V)$.

\smallskip

Since additivity holds almost surely for the Ito--integrals, the following equality follows
\begin{align}\label{locsteq1}
\begin{split}
\cdot (\omega\otimes_S\omega)(s,t)&=\cdot(\omega\otimes_S\omega)(0,t)-\cdot (\omega\otimes_S\omega)(0,s)\\
&-(-1)^{-\alpha}\int_s^t\int_0^s S(\xi-r)\cdot d\omega(r)\otimes_Vd\omega(\xi)
\end{split}
\end{align}
almost surely for $(s,t)\in \bar \Delta_{0,T}$. Note that $t\mapsto(\omega\otimes_S\omega)(0,\cdot)$ is continuous on a set of full measure.
Furthermore, the continuity of
\begin{equation*}
  \bar \Delta_{0,T}\ni (s,t)\mapsto \int_s^t\int_0^s S(\xi-r)\cdot d\omega(r)\otimes_Vd\omega(\xi)
\end{equation*}
follows by \cite{GLSch-local} by a fractional calculus argument based on fractional integrals.
By \eqref{locsteq1} $(\omega\otimes_S\omega)(s,t)$ is continuous on $\bar\Delta_{0,T}$ on a set of full measure.
Outside this set of full measure we set $(\omega\otimes_S\omega)\equiv 0$.
Hence $(\omega\otimes_S\omega)$ is continuous on $\bar\Delta_{0,T}$.

\end{proof}

\begin{lemma}\label{locall11}
Suppose that the conditions for $\nu,\,p$ of the last lemma hold. Then
\begin{align*}
\EE\|(\omega\otimes_S\omega)(s,t)-(\omega^n\otimes_S\omega^n)(s,t)\|_{L_2(L_2(V,V_\kappa),V\otimes V)}^{2p}& \le c_{p}\delta^\eps(t-s)^{2p-\eps}.
\end{align*}
\end{lemma}
\begin{proof}

Having Lemma \ref{stel5} in mind we obtain

\begin{align*}
  &\EE\|(\omega\otimes_S\omega)(s,t)-(\omega^n\otimes_S\omega^n)(s,t)\|_{L_2(L_2(V,V_\kappa),V\otimes V)}^{2p}\\
  &= \EE\bigg( \sum_{i,j,k}q_jq_k\lambda_i^{-2\kappa}((\omega\otimes_S \omega)_{ijk}-(\omega^n\otimes_S \omega^n)_{ijk})^2\bigg)^{p} \\
  & \le \bigg(\sum_{ijk}q_jq_k\lambda_i^\frac{-p\nu}{p-1} \bigg)^{p-1}\sum_{ijk}q_jq_k\lambda_i^{p\nu-2p\kappa}\EE((\omega\otimes_S \omega)_{ijk}-(\omega^n\otimes_S \omega^n)_{ijk})^{2p}\\
  &\le c \bigg(\sum_{ijk}q_j  q_k\lambda_i^\frac{-p\nu}{p-1} \bigg)^{p-1}\sum_{ijk}q_jq_k\lambda_i^{p(\nu-2\kappa)}  \delta^\eps\lambda_i(t-s)^{2p-\eps}.
\end{align*}

\end{proof}

Note that
\begin{equation*}
\int_\tau^t\int_s^\tau S(\xi-r)Ed\omega(r)\otimes_Vd\omega(\xi)=  \int_\tau^tS(\xi-\tau)\int_s^\tau S(\tau-r)Ed\omega(r)\otimes_Vd\omega(\xi),
\end{equation*}
and this is the reason to define the following operators:
\begin{align}\label{Somega}
\begin{split}
&e\in V\mapsto\omega_S(s,\tau)e:=(-1)^{-\alpha}\int_s^\tau(S(\xi-s)e)\otimes_Vd\omega(\xi)\\
&E\in L_2(V,V_\kappa) \mapsto S_{\omega}(\tau,t)E:=  \int_\tau^t S(t-r)Gd\omega(r).
\end{split}
\end{align}
We refer to Garrido-Atienza {\it et al.} \cite{GLSch-local} to check that these operators are well defined and for its additional properties.

\medskip

Now we can formulate the main result of this section:

\begin{lemma}\label{localt2}
Suppose that $\omega$ is a trace-class canonical Brownian--motion. Then Theorem \ref{stet1} holds.
\end{lemma}

\begin{proof}
First we  show that  $(\omega\otimes_S\omega)$ exists in the sense of Theorem \ref{stet1}. We apply the Garsia--Rodemich--Rumsey lemma for Deya {\it et al.} \cite{DeNeTi10}, Lemma 3.4., which is possible since $(\omega \otimes_S \omega)$ is continuous as we have already explained.
By Lemma \ref{locall110} we obtain that
\begin{equation*}
  \EE \int_0^T\int_0^t \frac{\|(\omega\otimes_S\omega)(s,t)\|_{L_2(L_2(V,V_\kappa),V\otimes V)}^{2p}}{|t-s|^{4\beta p+2}}dsdt\le
  c_p\int_0^T\int_0^t \frac{|t-s|^{2p}}{|t-s|^{4\beta p+2}}dsdt <\infty
\end{equation*}
when $p$ is sufficiently large since $\beta<1/2$.  This shows that there exists a set  $\Omega^\prime\subset \Omega$ of $\PP$--measure one on which $(\omega\otimes_S\omega)\in C_{2\beta} (\bar \Delta_{0,T}; L_2(L_2(V,V_\kappa),V\otimes V))$.

\smallskip

 To see the convergence property of the theorem we note that from Lemma \ref{locall11}
 \begin{align*}
   \EE&\int_0^T\int_0^t \frac{\|(\omega\otimes_S\omega)(s,t)-(\omega^n\otimes_S\omega^n)(s,t)\|_{L_2(L_2(V,V_\kappa),V\otimes V))}^{2p}}{|t-s|^{4\beta p+2}}dsdt\\
   &\le c_{p} \delta^\eps \int_0^T\int_0^t \frac{|t-s|^{2p-\eps}}{|t-s|^{4\beta p+2}}dsdt
   \le \delta^\eps c
 \end{align*}
for $p$ sufficiently large such that $p(2-4\beta)>1+\eps$. Recall that $\delta=\delta(n)=2^{-n}T$ and take a sequence $(o(n))_{n\in\NN}$ tending to zero for $n\to\infty$ such that $\sum_n\delta(n)^\eps/o(n)<\infty$. Then by the Chebyshev-lemma
 \begin{align*}
 \sum_n&\PP\bigg(\int_0^T\int_0^t \frac{\|(\omega\otimes_S\omega)(s,t)-(\omega^n\otimes_S\omega^n)(s,t)\|_{L_2(L_2(V,V_\kappa),V\otimes V)}^{2p}}{|t-s|^{4\beta p+2}}dsdt>o(n)\bigg)\\
 &\le
 \sum_n \frac{1}{o(n)}\EE\int_0^T\int_0^t \frac{\|(\omega\otimes_S\omega)(s,t)-(\omega^n\otimes_S\omega^n)(s,t)\|_{L_2(L_2(V,V_\kappa),V\otimes V)}^{2p}}{|t-s|^{4\beta p+2}}dsdt\\
 &\le c\sum_n\frac{\delta(n)^\eps}{o(n)}<\infty.
 \end{align*}
 Hence by the Borel--Cantelli--lemma we obtain the  convergence of the integrals inside the above probability to zero for $n\to\infty$ with probability 1.
In addition, we have
 \begin{equation}\label{Somegac}
   \lim_{n\to\infty}\sup_{0<s<r<t<T}\frac{\|\omega_S(r,t){S_{(\omega^n-\omega)}}(s,r)-{(\omega^n-\omega)_S}(r,t)S_{\omega^n}(s,r)\|_{L_2(L_2(V,V_\kappa),V\otimes V)}}{|r-s|^{\beta}|t-r|^{\beta}}=0,
 \end{equation}
see \cite{GLSch-local}. Now by Lemma 3.4 of \cite{DeNeTi10} we get
\begin{align*}
  \|(\omega & \otimes_S\omega)-(\omega^n\otimes_S\omega^n)\|_{C_{2\beta}(\bar \Delta_{0,T},L_2(L_2(V,V_\kappa),V\otimes V))}\\
  &\le c\bigg(\int_0^T\int_0^t \frac{\|(\omega\otimes_S\omega)(s,t)-(\omega^n\otimes_S\omega^n)(s,t)\|_{L_2(L_2(V,V_\kappa),V\otimes V)}^{2p}}{|t-s|^{4\beta p+2}}dsdt\bigg)^{1/2p}\\
  &+c\sup_{0<s<r<t<T}\frac{\|\omega_S(r,t){S_{(\omega^n-\omega)}(s,r)-(\omega^n-\omega)_S}(r,t)S_{\omega^n}(s,r)\|_{L_2(L_2(V,V_\kappa),V\otimes V)}}{|r-s|^{\beta}|t-r|^{\beta}}.
\end{align*}
Hence we have the convergence conclusion of Theorem \ref{stet1}  on a set $\Omega_0$ of full $\PP$--measure.

\end{proof}

\begin{remark}\label{localr1}
If $-A$ is given by the Laplace--operator with homogenous Dirichlet-- or Neumann--boundary conditions on a bounded smooth domain in $\RR^d$ such that $\lambda_i\sim i^{2/d}$ and $d=1$, the conditions of Lemma \ref{locall11} are fulfilled when choosing $\nu=1/2$, $\kappa>1/4$ and $p$ sufficiently large.
\end{remark}

\section{Additional properties of $(\omega\otimes_S\omega)$}
In this section we study additional properties of $(\omega\otimes_S\omega)$, which are valid for the two considered constructions.

Let us denote by $\Delta$ the set of pairs $(s,t)\in\RR^2$ such that $s\le t$. We consider the set $C(\Delta,L_2(L_2(V,V_\kappa),V\otimes V))$
equipped with a countable family of semi--norms for $m\in\NN$
\begin{equation*}
  \sup_{(s,t)\in \bar\Delta_{-m,m}}\|F(s,t)\|_{L_2(L_2(V,V_\kappa),V\otimes V)},
\end{equation*}
for $F\in C(\Delta,L_2(L_2(V,V_\kappa),V\otimes V))$. This gives us a Polish--space.
\smallskip

In the previous sections we found that we have $(\omega\otimes_S\omega)\in C(\bar\Delta_{0,T},L_2(L_2(V,V_\kappa),V\otimes V))$ on a set of probability one. Since our noise
is defined on $\RR$ we obtain that for any $m\in\NN$
\begin{equation*}
  (\omega\otimes_S\omega)\in C(\bar\Delta_{-m,m},L_2(L_2(V,V_\kappa),V\otimes V))
\end{equation*}
and hence that
\begin{equation*}
 (\omega\otimes_S\omega)\in C(\Delta,L_2(L_2(V,V_\kappa),V\otimes V))
\end{equation*}
on a set of probability one. On this set we have for every $m\in\NN$ that
$(\omega\otimes_S\omega)$ is $2\beta$ H{\"o}lder--continuous in the sense of \eqref{steeq102}  on $\bar\Delta_{-m,m}$. In addition, for every $m\in\NN$
\begin{equation*}
 \lim_{n\to\infty} \|(\omega\otimes_S\omega)-(\omega^n\otimes_S\omega^n)\|_{2\beta,-m,m}=0.
\end{equation*}
Outside this set let us define $(\omega^n\otimes_S\omega^n)\equiv 0$ such that those relations hold for all $\omega$.

 \begin{theorem}\label{stet2}
There exists a set of full measure where we have the Chen--equality: On a set of probability one for $-\infty <s\le\tau\le t<\infty$
 \begin{equation*}
  (\omega\otimes_S\omega)(s,t)-(\omega\otimes_S\omega)(\tau,t)-(\omega\otimes_S\omega)(s,\tau)=(-1)^{-\alpha} \omega_S(r,t)S_\omega(s,r).
\end{equation*}
 \end{theorem}
 It is a simple exercise to check this Chen--equality for $(\omega^n\otimes_S\omega^n)$ which follows easily by the semigroup properties and by the additivity of the integral with respect to the integral bounds. To see this inequality for $(\omega\otimes_S\omega)$ we can apply the convergence conclusion of
 Theorem \ref{stet1}.

\medskip

Finally, we show that $(\omega\otimes_S\omega)$ is strongly stationary. To do this we recall the definition of strong stationarity of the metric dynamical system $(\Omega,\fF,\PP,\theta)$.

\begin{definition}
The stochastic process $X=(X_t)_{t\in\RR}$ is called strongly stationary if there exists a $\theta$--invariant set $\tilde \Omega\in\fF$ of full measure such that for any $\tau\in\RR$ and $\omega\in \tilde \Omega$ we have that
\begin{equation}\label{steeq1}
  X_0(\theta_\tau\omega)=X_\tau(\omega).
\end{equation}
Such  a process is called {\em weakly stationary} if for any $\tau\in \RR$ there exists a set $\Omega_{\tau}\in\FF$ of full $\PP$--measure
such that for $\omega\in \Omega_{\tau}$ we have \eqref{steeq1}.
\end{definition}

To establish the last result of this paper we need to use the following theorem, see Lederer \cite{Led01}. Similar results can be found in Arnold and Scheutzow \cite{ArnSche98}, Imkeller and Lederer \cite{ImkLed02}, and K{\"u}mmel \cite{Kum14}.

\begin{theorem}\label{led}
Let $S$ be a Polish--space and let $X=(X_t)_{t\in\RR}$ be a stochastic process over the metric dynamical system $(\Omega,\fF,\PP,\theta)$. Suppose that $X$ is weakly stationary.
Then there exists a continuous stochastic process $\hat X=(\hat X)_{t\in\RR}$ such that $X$ and $\hat X$ are indistinguishable, and $\hat X$ is strongly stationary.
\end{theorem}

Now we apply this theorem.

\begin{theorem}\label{stet3}
The random fields $(\omega\otimes_S\omega)$ defined in Section \ref{stes2} and Section \ref{stes3} have a strongly stationary version
satisfying the convergence property of Theorem \ref{stet1} with respect to $\bar\Delta_{-m,m}$.
\end{theorem}

\begin{proof}
We know that $(\omega\otimes_S\omega)$ is continuous.
Since $\bar\Delta_{-m,m}$ is compact $(\omega\otimes_S\omega)$ restricted to this set is uniformly continuous and hence $\tau\mapsto(\omega\otimes_S\omega)(\cdot+\tau,\cdot+\tau)\in C(\bar\Delta_{-m,m},L_2(L_2(V,V_\kappa),V\otimes V))$.

In addition, shifting the integrals of the random field $(\omega\otimes_S\omega)$ we easily obtain that $\tau\mapsto(\omega\otimes_S\omega)(\cdot,\cdot)$ is weakly stationary: For every $\tau\in \RR$ we have that almost surely
\begin{equation*}
  (\omega\otimes_S\omega)(\cdot+\tau,\cdot+\tau)=(\theta_\tau\omega\otimes_S\theta_\tau\omega)(\cdot,\cdot).
\end{equation*}
In particular, the weak stationarity of $(\omega\otimes\omega)$ of Section 2 follows by Deya {\it et al.} \cite{DeNeTi10} Section 2, where this area is presented by a double integral of smooth processes, while for the second expression of \eqref{lastf} we use that $D^{1-\alpha}_{\cdot-}$ and $\theta_t$ commute in some sense.
Therefore, taking $S:=C(\Delta,L_2(L_2(V,V_\kappa),V\otimes V))$ we can apply Theorem \ref{led} to $\tau\mapsto X_\tau(\omega):=(\omega\otimes_S\omega)(\cdot+\tau,\cdot+\tau)\in S$, obtaining a version which is strongly stationary.
However this version is in addition indistinguishable such that we obtain the regularity,  the convergence properties and the Chen--equality for this new version of $(\omega\otimes_S\omega)$.

\end{proof}


\end{document}